\documentclass{article}
\usepackage{amsfonts, amssymb, latexsym, amscd}
\usepackage{amsmath, amsthm, epsfig, color, longtable}
\newtheorem{thm}{Theorem}[section]
\newtheorem{cor}[thm]{Corollary}

\newtheorem{prp}[thm]{Proposition}

\def\ds{$\hspace{-5pt}*\hspace{-1pt}$}
\def\1{{\bf{1}}}

\def\0{{\bf{0}}}

\def\to{{\!\top}}

\def\]*{\right]^{\mbox{\phantom{8}}}_{\mbox{\phantom{8}}}\hspace{-5pt}}

\begin{document}


\title{\bf Cospectrality results for signed graphs with two eigenvalues unequal to $\pm 1$}
\author{
Willem H. Haemers\thanks{\small{\tt{haemers@uvt.nl}}}
\\
{\small Dept. of Econometrics and O.R., Tilburg University, The Netherlands}
\\[5pt]
Hatice Topcu\thanks{\small{\tt{haticekamittopcu@gmail.com}}}
\\
{\small Dept. of Mathematics, Nev\c{s}ehir Hac{\i} Bekta\c{s} Veli University, T\"urkiye}\\
}
\date{}
\maketitle
\begin{abstract}
\noindent
Recently the collection $\cal G$ of all signed graphs for which the adjacency matrix has all
but at most two eigenvalues equal to $\pm 1$ has been determined.
Here we investigate $\cal G$ for cospectral pairs,
and for signed graphs determined by their spectrum (up to switching). 
If the order is at most 20, the outcome is presented in a clear table.
If the spectrum is symmetric we find all signed graphs in $\cal G$ determined by their spectrum, and we obtain all signed graphs
cospectral with the bipartite double of the complete graph. 
In addition we determine all signed graphs cospectral with the Friendship graph $F_\ell$, and show that there is no connected signed graph cospectral but not switching equivalent with $F_\ell$.
\\[5pt]
\end{abstract}

\section{Introduction}

A {\em signed graph} $G^\sigma$ is a graph $G=(V, E)$ together with
a function $\sigma : E \rightarrow \{-1, +1\}$. 
So, every edge is either positive or negative.
The graph $G$ is called the {\em underlying graph} of $G^\sigma$.
The adjacency matrix $A$ of $G^\sigma$ is obtained from the adjacency matrix of $G$,
by replacing $1$ by $-1$ whenever the corresponding edge is negative.
The signed graph $G^{-\sigma}$ with adjacency matrix $-A$ is called the {\em negative} of $G^\sigma$.
The spectrum of $A$ is also called the spectrum of the signed graph $G^\sigma$.
For a vertex set $X\subset V$, the operation that changes the sign of all edges between $X$ and $V\setminus X$ is called switching.
In terms of the matrix $A$, switching multiplies the rows and columns of $A$ corresponding to $X$ by $-1$.
If a signed graph can be switched into an isomorphic copy of another signed graph, the two signed graphs are called
{\em switching isomorphic}.
Switching isomorphic signed graphs have similar adjacency matrices and therefore they are cospectral (that is, they have the same spectrum).
We define a signed graph $G^\sigma$ to be {\em determined by its spectrum up to switching} (DSS for short), if every signed graph cospectral with $G^\sigma$ is switching isomorphic to $G^\sigma$.
It is obvious that a signed graph is DSS
if and only if its negative is DSS.
For more about the spectra of signed graphs we refer to~\cite{BCKW}.
\\

Let $\cal G$ be the set of signed graphs with at most two eigenvalues unequal to $\pm 1$.
All signed graphs in $\cal G$ have been determined. 
The unsigned graphs in $\cal G$ were found by Cioab\u{a}, Haemers, Vermette and Wong in \cite{CHVW}.
The remaining signed graphs in $\cal G$ were determined by the present authors in \cite{HT,HT'}, and independently, by Wang, Hou and Li in \cite{WHL}.  
If a signed graph $G^\sigma\in{\cal G}$ has all eigenvalues 
at least $-1$, or at most $1$, then $G^\sigma$ or its negative $G^{-\sigma}$ is switching isomorphic to the disjoint union of unsigned complete graphs, which is DSS (see~\cite{HT}).
Therefore we restrict to the subset $\cal G'$ of signed graphs in $\cal G$ with one eigenvalue $r>1$ and one eigenvalue $s<-1$. 
If $\chi(x)$ is the characteristic polynomial of a signed graph $G^\sigma$,
then $G^\sigma\in{\cal G'}$ if and only if
$$
\chi(x)=(x^2 - ax - b)(x-1)^f (x+1)^g,
$$
for integers $a$, $b>|a|+1$, $f$ and $g$.
Then $a=r+s$, $b=-rs$, $f+g+2=n=|V|$, and $f-g=a$ (since trace$(A)=0$).
We will call $(a,b,n)$ the {\em characteristic triple} of $G^\sigma$.
Thus we have that two signed graphs in $\cal G'$ are cospectral if and only if they have the same characteristic triple.
Note that $\cal G'$ (as well as $\cal G$) is closed under switching, taking the negative, and the addition or deletion of isolated edges.
The signed graphs in $\cal G'$ are presented in Table~\ref{classification} in such a way that each $G^\sigma\in{\cal G'}$ can be obtained from a signed graph from Table~\ref{classification} by switching, taking the negative, or adding a number of isolated edges.
We use $O$ and $J$ for the all-zeros, and the all-ones matrices, respectively; $I_n$ and $R_n$ are the 
identity and the reverse identity matrix of order $n$
(i.e. $(R_n)_{i,j}=1$ if $i+j=n+1$ and $0$ otherwise).  
The last column gives the characteristic triples of the presented signed graphs.

\begin{center}
{\small
\begin{longtable}[t]{|l|l|l|}
\hline
adjacency matrix\phantom{$8^{8^8}$}
&
parameter restrictions
&
~~characteristic triple
\\[5pt]
\hline
$A_{0} =
\left[
\begin{array}{cc}
\!\!\!J-I_m & J
\\
J &\! -J +I_\ell
\end{array}
\]*$
&
$m\geq\ell\geq 2$
&
$~~(~m-\ell,~~2m\ell-m-\ell+1,~~m+\ell~)$
\\ 
\hline
$A_{1}
=\left[
\begin{array}{cc}
\!J\!-\!I_m\! & J \\ J & \!-R_{2\ell}\!
\end{array}
\]*$
&
$m\geq 1,~\ell\geq 2$
&
$~~(~m-2,~~2m\ell+m-1,~~m+2\ell~)$
\\
\hline
$
A_2
=\left[
\begin{array}{cc}
\!R_{2m}\! & J \\ J & \!-R_{2\ell}\!
\end{array}
\]*$
&
$m\geq\ell\geq 2$
&
$~~(~0,~~4m\ell+1,~~2m+2\ell~)$
\\
\hline
$
A_3
=\left[
\begin{array}{rrcr}
\!J\!-\!I_m\! & \1     & \1  & O\ \\
\ \1^\top\    & 0      &  1  & -\1^\top \\
\ \1^\top\    & 1      &  0  &  \1^\top \\
O\ \          & \!-\1  & \1  & \!I_\ell-\!J\!
\end{array}
\]*$
&
$m\geq\ell\geq 1$
&
$~~(~m-\ell,~~(m+1)(\ell+1),~~m+\ell+2~)$
\\
\hline
$
A_4=\left[
\begin{array}{cccc}
\!J\!-\!I_m\! & J                & \, J  & O \\
J\ \          & \!I_\ell\!-\!J\! & \, O  & J \\
J\ \          & O                & R_2\! & O \\
O\ \          & J                & \, O  & \!-R_2\! 
\end{array}
\]*$
&
$m\geq\ell\geq 1$
&
$~~(~m-\ell,~~2m\ell+m+\ell+1,~~m+\ell+4~)$
\\
\hline
$
A_5
=\left[
\begin{array}{ccc}
\! J\! -\! I_m \! & J            & J \\
J           &\! J\!-\!I_\ell \!& O \\
J           & O            &\! I_k\!-\! J \! \\
\end{array}
\]*$
&
$(m,\ell)=
\left\lbrace\!
\begin{array}{l}
(3,8)\\(4,6)\\(6,5)
\end{array}
\right.
,~k\geq 1$
&
$
\begin{array}{l}
(~9-k,~~11k+10,~~k+11~)
\\
(~8-k,~~11k+9,~~k+10~)
\\
(~9-k,~~14k+10,~~k+11~)
\end{array}
$
\\
\hline

$
A_6
=\left[
\begin{array}{ccc}
\! J\!-\! I_m\! & J        & J \\
J     & \! I_\ell\!-\! J \! & O \\
J     & O        & \! R_{2k}\!
\end{array}
\]*$
&
$m\geq 1,~(\ell,k)=
\left\lbrace
\hspace{-3pt}
\begin{array}{l}
(3,4)
\\
(4,3)
\end{array}
\right.$
&
$
\begin{array}{l}
(~m-1,~~11m+2,~~m+11~)
\\
(~m-2,~~11m+3,~~m+10~)
\end{array}
$
\\
\hline
$
A_7=\left[
\begin{array}{ccc}
R_{2m}& J        & J \\
J     & \! J\!-\! I_\ell\! & O \\
J     & O        & \! -R_{2k}\!
\end{array}
\]*$
&
$m\geq 1,~(\ell,k)=
\left\lbrace
\hspace{-3pt}
\begin{array}{l}
(3,3)
\\
(4,2)
\end{array}
\right.$
&
$
\begin{array}{l}
(~1,~~18m+2,~~2m+9~)
\\
(~2,~~16m+3,~~2m+8~)
\end{array}
$
\\
\hline
 
$A_{8}=\left[
\begin{array}{cccc}
R_2 & J & \1  & O \\
J   & \!I_m\!-\!J\! & \0  & J \\
\1^\to\! & \0^\to & 0 & \0^\to \\
O & J & \0 & -R_4 
\end{array}
\]*$
&
$m\geq 1$ 
&
$~~(~1-m,~~6m+2,~~m+7~)$
\\
\hline
$
A_{9}=\left[
\begin{array}{cccc}
R_{2m} & J & J & \0 \\
J   & R_{2} & O & \1 \\
J & O & -R_2 & \0 \\ 
\0^\to & \1^\to & \0^\to & 0 
\end{array}
\]*$
&
$m\geq 1$ 
&
$~~(~1,~~8m+2,~~2m+5~)$
\\ 
\hline
$A_{10}=\left[
\begin{array}{cccc}
\! J\!-\! I_m\!  & J & O & O \\
J & O  & J & \!J\!-\! I_\ell\! \\
O & J & \!I_m\!-\!J  & O \\ 
O & \! \! J\!-\!I_\ell\! & O & O 
\end{array}
\]*$
&
$(m,\ell)=
\left\lbrace
\hspace{-3pt}
\begin{array}{l}
(3,4)
\\
(4,3)
\end{array}
\right.
$
&
$
\begin{array}{l}
(~0,~~36,~~14~)
\\
(~0,~~36,~~14~)
\end{array}
$
\\
\hline
$A_{11}=\left[
\begin{array}{cccc}
\! J\!-\!I_m\!  & J & J & O \\
J & \! I_m\!-\!J\!  & O & J  \\
J & O & O & \! J\!-\!I_\ell\! \\
O & J & \! J\!-\!I_\ell\! & O
\end{array}
\]*$
&
$(m,\ell)=
\left\lbrace
\hspace{-3pt}
\begin{array}{l}
(3,4)
\\
(4,3)
\end{array}
\right.
$
&
$\begin{array}{l}
(~0,~~45,~~14~)
\\
(~0,~~52,~~14~)
\end{array}
$
\\
\hline

$A_{12}=\left[
\begin{array}{cccc}
\! J\!-\! I_m\! & J & J & O \\
J & \! I_\ell\!-\!J\!     & O & J \\
J & O &   \! J\!-\!I_k\!    & J\\
O & J & J & \! I_j\!-\!J\!\!
\end{array}
\]*$
&
$(m,\ell,k,j)=
\left\lbrace
\hspace{-3pt}
\begin{array}{l}
(4,4,4,4)\\
(6,3,3,6)\\
(6,4,3,4)\\
(6,6,3,3)
\end{array}
\right.
$
&
$\begin{array}{l}
(~0,~~81,~~16~)\\
(~0,~~109,~~18~)\\
(-1,~~92,~~17)\\
(~0,~~100,~~18~)
\end{array}$
\\
\hline
$
A_{13}=\left[
\begin{array}{cccc}
\! J\!-\! I_m\! & J              & O               & \0 \\
J               & \!-R_{2\ell}\! & J               & \1 \\
O               & J              & \! I_4\!-\! J\! & \0 \\ 
\0^\to          & \1^\to         & \0^\to          & 0 
\end{array}
\]*$
&
$(m,\ell)=
\left\lbrace
\hspace{-3pt}
\begin{array}{l}
(5,3)
\\
(6,2)
\end{array}
\right.
$
&
$\begin{array}{l}
(~0,~~72,~~16~)
\\
(~1,~~60,~~15~)
\end{array}
$
\\
\hline

$
A_{14}=\left[
\begin{array}{ccc}
\! J\!-\! I_m\! & J & O \\
J & \!-R_{2\ell}\! & J \\
O & J & \!-R_{4}\!
\end{array}
\]*$
&
$(m,\ell)=
\left\lbrace
\hspace{-3pt}
\begin{array}{l}
(5,3)
\\
(6,2)
\end{array}
\right.
$
&
$\begin{array}{l}
(~1,~~62,~~15~)
\\
(~2,~~51,~~14~)
\end{array}
$
\\
\hline

$A_{15}=\left[
\begin{array}{cccc}
\! J\!-\!I_m\!  & J & J & O \\
J & \! I_\ell\!-\!J\! & O & J \\
J & O & \!R_{2k}\! & J \\
O & J & J & \! -R_{2j}\!\!
\end{array}
\]*$
&
$(m,\ell,k,j)=
\left\lbrace
\hspace{-3pt}
\begin{array}{l}
(3,3,3,3)\\
(4,3,3,2)\\
(4,4,2,2)
\end{array}
\right.
$
&
$
\begin{array}{l}
(~0,~~85,~~18~)\\
(~1,~~78,~~17~)\\
(~0,~~73,~~16~)
\end{array}
$
\\
\hline

$A_{16}=\left[
\begin{array}{cccc}
R_2 & J & J & O \\
J & -R_2 & O & J  \\
J & O & O & \!I_3\! \\
O & J & \!I_3\! & O
\end{array}
\]*
$
&&
$~~(~0,~~17,~~10~)$
\\
\hline

$
A_{17}=\left[
\begin{array}{cccccc}
R_2 & J & J & \1 & O &  \0 \\[-2pt]
J & -R_2 & O & \0 & J &  \1 \\[-2pt] 
J & O & R_2 & \0 &  J & \0 \\ 
\1^\to & \0^\to & \0^\to & 0 & \0^\to  & 0 \\[-2pt]  
O & J & J & \0 & -R_2 & \0  \\
\0^\to & \1^\to & \0^\to & 0 & \0^\to & 0 \\  
\end{array}
\]*$
&&
$~~(~0,~~20,~~10~)$
\\
\hline

$
A_{18}=\left[
\begin{array}{cccc}
\! J\!-\!I_m\! & J & \1 & O \\
J   & -R_{2\ell} & \0 & J \\
\1^\to & \0^\to & 0 & \0^\to \\ 
O & J & \0 & -R_2 
\end{array}
\]*$
&
$(m,\ell)=
\left\lbrace
\hspace{-3pt}
\begin{array}{l}
(3,4)
\\
(4,3)
\end{array}
\right.
$
&
$\begin{array}{l}
(~0,~~45,~~14~)
\\
(~1,~~44,~~13~)
\end{array}
$
\\
\hline

$
A_{19}=\left[
\begin{array}{ccccc}
R_2 & J & J & \1 & O \\
J & \! I_m\!-\!J\! & O & \0 & J \\
J & O & \!R_{2\ell}\! & \0 & J \\
\1^\to & \0^\to & \0^\to & 0 & \0^\to \\
O & J & J & \0 & -R_2 \\
\end{array}
\]*$
&
$(m,\ell)=
\left\lbrace
\hspace{-3pt}
\begin{array}{l}
(3,3)
\\
(4,2)
\end{array}
\right.
$
&
$\begin{array}{l}
(~0,~~40,~~14~)
\\
(-1,~~38,~~13~)
\end{array}
$
\\
\hline

$A_{20}=\left[
\begin{array}{cc}
\!J\!-\!I_m\! & J \\ J & \!R_{2\ell}\!
\end{array}
\]*$
&
$m\geq 2,~\ell\geq 2$
&
$~~(~m,~~2m\ell-m+1,~~m+2\ell~)$
\\
\hline

$
A_{21}=\left[
\begin{array}{cc}
\!R_{2m}\! & J \\ J & \!R_{2\ell}\!
\end{array}
\]*$
&
$m\geq\ell\geq 2$
&
$
~~(~2,~~4m\ell-1,~~2m+2\ell~)$
\\
\hline

$
A_{22}=\left[
\begin{array}{cc}
O & \!J\!-\!I_m\! \\ \!J\!-\!I_m\! & O
\end{array}
\]*$
&
$m\geq 3$
&
$~~(~0,~~(m-1)^2,~~2m~)$
\\
\hline

$A_{23}=\left[
\begin{array}{cccc}
O & O & \!J\!-\!I_3\! & J \\
O & \!O & O & \!J\!-\!I_3\! \\
\!J\!-\!I_3\! & O & O & O \\
J & \!J\!-\!I_3\! & O & \!O \\
\end{array}
\]*$
&&
$~~(~0,~~16,~~12~)$
\\
\hline

$A_{24}=\left[
\begin{array}{cccc}
0 & \0^\top & 1 & \1^\top \\
\0 & O & \1 & I_4 \\
1 & \1^\top & 0 & \0^\top \\
\1 & I_4 & \0 & O 
\end{array}
\]*$
&&
$~~(~0,~~9,~~10~)$
\\
\hline

$A_{25}=
\left[
\begin{array}{ccc}
\! J\! -\! I_m \! & J & O \\
J & O & \! J\! -\! I_\ell \! \\
O & \! J\! -\! I_\ell \! & 0 
\end{array}
\]*$
&
$(m,\ell)=
\left\lbrace
\hspace{-3pt}
\begin{array}{l}
(3,5)
\\
(4,4)
\end{array}
\right.
$
&
$\begin{array}{l}
(~1,~~32,~~13~)
\\
(~2,~~27,~~12~)
\end{array}
$
\\
\hline

$
A_{\infty}=\left[
\begin{array}{cc}
\!J\!-\!I_m\! & O \\ O & \!I_\ell\!-\!J\!
\end{array}
\]*$
&
$m\geq\ell\geq 3$
&
$~~(~m-\ell,~~(m-1)(\ell-1),~~m+\ell~)$
\\
\hline
\caption{{  The signed graphs in $\cal G'$\hspace{-15pt}}\phantom{$8^{8^8}$}}\label{classification}
\end{longtable}
}
\end{center}

Table~\ref{classification} is obtained by combining the results from \cite{CHVW,HT,HT'}.
The matrices $A_1$ to $A_{19}$ are the same as in \cite{HT'} (after correction of a typo in the eigenvalues of $A_7(m,3,3)$),
$A_0$ is a signed complete graph called $\widetilde{J}$ in \cite{HT'}, 
$A_{20}$ to $A_{25}$ are the unsigned graphs in $\cal G'$ obtained in \cite{CHVW}, 
and $A_\infty$ is the disjoint union of 
two complete graphs, one with all signs positive and one with all signs negative.
The original descriptions had some overlap, which we have removed here.
Case $(v)$ of Theorem~1 from \cite{CHVW} has been removed because it is of type $A_3$ with $k=1$. 
Also $A_{20}$ with $m=1$ is removed because it is switching isomorphic with the negative of $A_1$, 
and for $A_0$, $A_2$, $A_3$, $A_4$, $A_{21}$, and $A_\infty$ the cases with $m<\ell$ are removed because they are switching isomorphic to the negatives with $m$ and $\ell$ interchanged.
\\

The aim of this paper is to find the cospectral signed graphs for every signed graph in $\cal G'$, and decide which ones are DSS.
However, it turned that there are many cospectral coincidences, which makes a comprehensive discription too messy to be usefull. 
Therefore we decided to restrict to the signed graphs in $\cal G'$ with at most 20 vertices,
and in addition, to pay attention to a number of interesting special cases.

Up to switching, taking the negative and ignoring isolated edges there are almost 600 signed graphs in $\cal G'$ with at most 20 vertices.
For each of these we indicate if they are DSS and we give all cospectral signed graphs if they are not DSS;
see Section~\ref{20}.

If $G^\sigma\in{\cal G'}$ has a characterisic triple $(a,b,n)$ with $a=0$, then the spectrum of $G^\sigma$ is symmetric, which means that it is invariant under multiplication by $-1$.
Unsigned graphs with symmetric spectrum are bipartite, and bipartite signed graphs are switching isomorphic to their negatives.
This motivated the search for 
signed graphs with symmetric spectrum which are not switching isomorphic to their negatives
(see~\cite{BCKW,GHMP,GS}).
Here, in Section~\ref{symmetry} we find many signed graphs with this property including some which are cospectral with a bipartite graph.  
We also find all signed graphs in $\cal G'$ with symmetric spectrum which are DSS, 
and we determine all signed graphs cospectral with the bipartite double of the complete graph.

The original motivation for the research of this paper 
comes from the determination of the unsigned graphs in $\cal G'$~\cite{CHVW}.
This has led to the spectral characterization of many (unsigned) graphs including the friendship graphs
(consisting of a number of edge-disjoint triangles meeting in one vertex), which had been an open problem for a number of years.
In Section~\ref{friendship} we generalize this result and determine which friendship graphs are DSS.

\section{Cospectrality}

We start with some infinite families of cospectral signed graphs in $\cal G'$.
The following proposition follows easily from Table~\ref{classification}.
(We use \lq$+${\rq} for the disjoint union
of two signed graph, and we identify a signed graph with its adjacency matrix.)

\begin{prp} \label{cosp}
The following pairs of signed graphs are cospectral, and so are their negatives.
\\
$(i)$ 
$A_4 (m,\ell)$, and $A_0(m+1,\ell+1)+K_2$
($m,\ell\geq 1$),
\\
$(ii)$
$A_\infty(m+2,\ell+2)$, and $A_3(m,\ell)+K_2$ 
($m,\ell\geq 1$),
\\
$(iii)$
$A_\infty(m,m)$, and $A_{22}(m)$ ($m\geq 3$).
\end{prp}

Note that item $(iii)$ is a special case of the more general result that the disjoint union of a signed graph and its negative is cospectral with its bipartite double.
From Proposition~\ref{cosp} it follows that none of the signed graphs $A_4$, $A_{22}$ and $A_\infty$ are DSS.
But it was proved in~\cite{HT} and in~\cite{WHL} that every signed graph with matrix $A_0$ is DSS.
By inspection of Table~\ref{classification} we can find many other families of cospectral signed graphs in $\cal G'$. 
Sometimes cospectral signed graphs occur within one type of matrix, for example:

\begin{prp}
The signed graphs $A_2(m,\ell)+\alpha K_2$ and 
$A_2(m',\ell')+\alpha'K_2$ are cospectral if and only if $m\ell=m'\ell'$ and $m+\ell+\alpha = m'+\ell'+\alpha'$.
\end{prp}

This proposition shows that a set of signed graphs in $\cal G'$ that are mutually cospectral but not switching isomorphic can have arbitrary large cardinality.
Indeed, take for example $m=\ell=2^k$, then there are $k+1$ different choices for $m'$ and $\ell'$ so the proposition gives $k+1$ mutually cospectral signed graphs of order $2+2^{2k+1}$.  

For a given signed graph $G^\sigma\in\cal G'$ with characteristic triple $(a,b,n)$ it is not hard to find all signed graphs cospectral with $G^\sigma$, and thus to decide if $G^\sigma$ is DSS.
We just have to search in Table~\ref{classification} for characteristic triples $(a',b',n')$ with $a'=\pm a$, $b'=b$ and $n'\leq n$.
If $n=n'$, the corresponding signed graphs (when $a=a'$) or their negatives  (when $a'=-a$) are cospectral with $G^\sigma$.
If $n'<n$ we need to add $(n-n')/2$ isolated edges.
We have worked this out for all signed graphs with $n\leq 20$.

\section{The signed graphs in $\cal G'$ of order at most $20$}

We generated the characteristic triples of all signed graphs given in Table~\ref{classification} with at most 20 vertices.
Note that this includes all cases 
which are not part of an infinite family.
The outcome is presented in Table~\ref{list} in the following way:
Each entry consists of a 4-tuple $(a,b,n,A_i)$ consisting of
the characteristic triple followed by the corresponding matrix type.
If $a<0$ the 4-tuple is replaced by its negative: $(-a,b,n,-A_i)$. 
If $a=0$ and the corresponding signed graph is not switching isomorphic to its negative we get two 
4-tuples: $(0,b,n,A_i)$, and $(0,b,n,-A_i)$
(see next section for more about the case $a=0$).
To save space we do not write the values of the parameters. 
Only $A_{10}(3,4)$ and $A_{10}(4,3)$ give identical 4-tuples: $(0,36,14,A_{10})$; in all other cases the value of the parameters can easily be deduced.
The 4-tuples have been ordered lexicographically. 
This way all 4-tuples with the same $a$ and $b$ are clustered together as consequitive entries with nondecreasing orders.
In the table the different clusters are separated by a line.
For each signed graph from such a cluster
the cospectral signed graphs are easily found. 
They correspond to the 4-tuples in that cluster with the same order, together with the ones in that cluster with smaller order extended with an appropriate number of isolated edges. 
Thus a signed graph belonging to a 4-tuple from the table is DSS if it is the first one in the corresponding cluster,
and there is no other 4-tuple in that cluster with the same order. 
In Table~\ref{list} we marked these 4-tuples with a $*$.
Thus we have

\begin{thm}\label{20}
A signed graph $G^\sigma\in{\cal G'}$ of order at most 20 with no isolated edges is DSS if and only if $G^\sigma$, or its negative $G^{-\sigma}$ corresponds to a 4-tuple in Table~\ref{list} marked with a $*$.
\end{thm}

Note that a signed graph $G^\sigma\in{\cal G'}$ with isolated edges can be DSS.
This is the case if $G^\sigma$ with the isolated edges removed is DSS, and every other signed graph with the same $|a|$ and $b$ in its characteristic triple has larger order than $G^\sigma$.
For example $A_2(3,3)+\alpha K_2$
with characteristic triple $(0,37,12+2\alpha)$ is DSS
if $0\leq\alpha\leq 3$.

By inspection of Table~\ref{list} we see some interesting sets of mutually cospectral signed graphs.
For example there are eleven characteristic triples $(a,b,n)$ with
$a=0$, $b=25$ and $n\in\{8,10,12,14\}$.
This means that there are eleven signed graphs with characteristic triple  
$(0,25,14)$ which are mutually cospectral,
but not switching isomorphic.
We also see that there are five mutually cospectral signed graphs with triple $(0,25,10)$,
of which four are connected.

\begin{center}
\begin{longtable}[t]{c}
\begin{tabular}{|l|l|l|l|l|l|}
\hline
\multicolumn{6}{|l|}{}
\\
\cline{1-1}\cline{3-5}
\ds(0,4,4,A$_{3}$)&(0,29,16,-A$_{1}$)&(0,72,16,A$_{13}$)&\ds (1,12,7,A$_{3}$)&\ds(1,58,19,A$_{9}$)&(2,27,12,A$_{25}$)\\[1pt] \cline{2-2} \cline{5-5}
(0,4,6,A$_{22}$)&(0,33,12,A$_{2}$)&(0,72,16,-A$_{13}$)&(1,12,9,A$_{\infty}$)&\ds(1,60,15,A$_{13}$)&(2,27,16,A$_{20}$)\\[1pt]
\cline{3-3}\cline{5-6}
(0,4,6,A$_{\infty}$)&(0,33,12,-A$_{2}$)&\ds(0,73,16,A$_{15}$)&(1,12,13,-A$_{1}$)&\ds(1,62,15,A$_{14}$)&\ds(2,31,12,A$_{21}$)\\[1pt]
\cline{1-1}\cline{4-5}
\ds(0,5,4,A$_{0}$)&(0,33,18,A$_{1}$)&(0,73,18,A$_{2}$)&\ds(1,14,7,A$_{1}$)&\ds(1,72,13,A$_{0}$)&(2,31,18,A$_{20}$)\\[1pt]
\cline{6-6}
(0,5,6,A$_{4}$)&(0,33,18,-A$_{1}$)&(0,73,18,-A$_{2}$)&(1,14,9,-A$_{8}$)&(1,72,15,A$_{4}$)&(2,35,12,A$_{1}$)\\[1pt]
\cline{1-3}
(0,8,8,A$_{8}$)&\ds(0,36,12,A$_{3}$)&\ds(0,81,16,A$_{12}$)&(1,14,11,-A$_{6}$)&(1,72,17,A$_{3}$)&(2,35,12,A$_{3}$)\\[1pt]
(0,8,8,-A$_{8}$)&(0,36,14,A$_{10}$)&(0,81,18,A$_{2}$)&(1,14,15,-A$_{1}$)&(1,72,19,A$_{\infty}$)&(2,35,12,A$_{7}$)\\[1pt]
\cline{1-1} \cline{4-5}
(0,9,6,A$_{1}$)&(0,36,14,A$_{10}$)&(0,81,18,-A$_{2}$)&\ds(1,16,17,-A$_{1}$)&\ds(1,74,17,A$_{7}$)&(2,35,12,A$_{21}$)\\[1pt]
\cline{4-5}
(0,9,6,-A$_{1}$)&(0,36,14,A$_{22}$)&(0,81,18,A$_{3}$)&\ds(1,18,7,A$_{0}$)&\ds(1,78,17,A$_{15}$)&(2,35,14,A$_{6}$)\\[1pt]
\cline{5-5}
(0,9,6,A$_{3}$)&(0,36,14,A$_{\infty}$)&(0,81,20,A$_{22}$)&(1,18,9,A$_{4}$)&\ds(1,86,17,A$_{5}$)&(2,35,14,A$_{\infty}$)\\[1pt]
\cline{2-2}\cline{5-5}
(0,9,8,A$_{22}$)&\ds(0,37,12,A$_{2}$)&(0,81,20,A$_{\infty}$)&(1,18,9,A$_{9}$)&\ds(1,90,19,A$_{3}$)&(2,35,20,A$_{20}$)\\[1pt]
\cline{3-3}\cline{5-6}
(0,9,8,A$_{\infty}$)&(0,37,20,A$_{1}$)&\ds(0,85,14,A$_{0}$)&(1,18,19,-A$_{1}$)&\ds(1,92,17,-A$_{12}$)&\ds(2,39,10,A$_{0}$)\\[1pt]
\cline{4-4}
(0,9,10,A$_{24}$)&(0,37,20,-A$_{1}$)&(0,85,16,A$_{4}$)&(1,20,9,A$_{1}$)&(1,92,19,A$_{7}$)&(2,39,12,A$_{4}$)\\[1pt]
\cline{1-2}\cline{5-5}
\ds(0,13,6,A$_{0}$)&(0,40,14,A$_{19}$)&(0,85,18,A$_{15}$)&(1,20,9,A$_{3}$)&\ds(1,98,15,A$_{0}$)&(2,39,14,A$_{21}$)\\[1pt]
\cline{6-6}
(0,13,8,A$_{1}$)&(0,40,14,-A$_{19}$)&(0,85,20,A$_{2}$)&(1,20,11,A$_{7}$)&(1,98,17,A$_{4}$)&\ds(2,43,14,A$_{1}$)\\[1pt]
\cline{2-2}\cline{6-6}
(0,13,8,-A$_{1}$)&\ds(0,41,10,A$_{0}$)&(0,85,20,-A$_{2}$)&(1,20,11,A$_{\infty}$)&(1,98,19,A$_{5}$)&\ds(2,47,14,A$_{6}$)\\[1pt]
\cline{3-5}
(0,13,8,A$_{4}$)&(0,41,12,A$_{4}$)&(0,97,18,A$_{5}$)&\ds(1,24,13,A$_{6}$)&\ds(1,108,19,-A$_{5}$)&(2,47,14,A$_{21}$)\\[1pt]
\cline{4-5}
(0,13,12,A$_{6}$)&(0,41,14,A$_{2}$)&(0,97,18,-A$_{5}$)&(1,26,11,A$_{1}$)&\ds(1,122,19,A$_{5}$)&(2,47,16,A$_{21}$)\\[1pt]
\cline{5-6}
(0,13,12,-A$_{6}$)&(0,41,14,-A$_{2}$)&(0,97,20,A$_{2}$)&(1,26,11,A$_{9}$)&\ds(1,128,17,A$_{0}$)&\ds(2,48,14,A$_{3}$)\\[1pt]
\cline{1-2} \cline{4-4}
\ds(0,16,8,A$_{3}$)&(0,45,14,A$_{11}$)&(0,97,20,-A$_{2}$)&\ds(1,30,11,A$_{3}$)&(1,128,19,A$_{4}$)&(2,48,16,A$_{\infty}$)\\[1pt]
\cline{3-3}\cline{5-6}
(0,16,10,A$_{22}$)&(0,45,14,A$_{18}$)&\ds(0,100,18,A$_{12}$)&(1,30,13,A$_{\infty}$)&\ds(1,162,19,A$_{0}$)&(2,51,14,A$_{7}$)\\[1pt]
\cline{4-5}
(0,16,10,A$_{\infty}$)&(0,45,14,-A$_{18}$)&(0,100,20,A$_{3}$)&\ds(1,32,9,A$_{0}$)&\ds(2,7,6,A$_{20}$)&(2,51,14,A$_{14}$)\\[1pt]
\cline{2-3}\cline{5-5}
(0,16,12,A$_{23}$)&(0,49,14,A$_{2}$)&\ds(0,101,20,A$_{2}$)&(1,32,11,A$_{4}$)&\ds(2,8,6,A$_{3}$)&(2,51,16,A$_{1}$)\\[1pt]
\cline{1-1}\cline{3-3}\cline{6-6}
\ds(0,17,8,A$_{2}$)&(0,49,14,-A$_{2}$)&\ds(0,109,18,A$_{12}$)&(1,32,13,A$_{1}$)&(2,8,8,A$_{\infty}$)&\ds(2,55,18,A$_{21}$)\\[1pt]
\cline{5-6}
(0,17,10,A$_{1}$)&(0,49,14,A$_{3}$)&(0,109,20,A$_{5}$)&(1,32,13,A$_{25}$)&\ds(2,11,6,A$_{0}$)&\ds(2,59,12,A$_{0}$)\\[1pt]
\cline{4-4}
(0,17,10,-A$_{1}$)&(0,49,16,A$_{2}$)&(0,109,20,-A$_{5}$)&\ds(1,34,13,A$_{9}$)&(2,11,8,A$_{4}$)&(2,59,14,A$_{4}$)\\[1pt]
\cline{3-4}
(0,17,10,A$_{16}$)&(0,49,16,-A$_{2}$)&\ds(0,113,16,A$_{0}$)&\ds(1,36,13,A$_{6}$)&(2,11,8,A$_{20}$)&(2,59,16,A$_{21}$)\\[1pt]
\cline{1-1}\cline{4-5}
\ds(0,20,10,A$_{17}$)&(0,49,16,A$_{22}$)&(0,113,18,A$_{4}$)&(1,38,13,A$_{7}$)&(2,15,8,A$_{3}$)&(2,59,18,A$_{1}$)\\[1pt]
\cline{1-1}\cline{3-3}\cline{6-6}
(0,21,12,A$_{1}$)&(0,49,16,A$_{\infty}$)&(0,136,20,A$_{5}$)&(1,38,13,-A$_{19}$)&(2,15,8,A$_{21}$)&(2,63,16,A$_{3}$)\\[1pt]
\cline{2-2}
(0,21,12,-A$_{1}$)&\ds(0,52,14,A$_{11}$)&(0,136,20,-A$_{5}$)&(1,38,15,A$_{1}$)&(2,15,10,A$_{20}$)&(2,63,16,A$_{21}$)\\[1pt]
\cline{1-4}
\ds(0,25,8,A$_{0}$)&(0,57,18,A$_{2}$)&\ds(0,145,18,A$_{0}$)&\ds(1,42,13,A$_{3}$)&(2,15,10,A$_{\infty}$)&(2,63,18,A$_{\infty}$)\\[1pt]
\cline{5-5}
(0,25,10,A$_{2}$)&(0,57,18,-A$_{2}$)&(0,145,20,A$_{4}$)&(1,42,15,A$_{9}$)&\ds(2,19,8,A$_{1}$)&(2,63,20,A$_{21}$)\\[1pt]
\cline{2-3}\cline{6-6}
(0,25,10,-A$_{2}$)&\ds(0,61,12,A$_{0}$)&\ds(0,181,20,A$_{0}$)&(1,42,15,A$_{\infty}$)&(2,19,10,A$_{7}$)&\ds(2,67,16,A$_{7}$)\\[1pt]
\cline{3-4}
(0,25,10,A$_{3}$)&(0,61,14,A$_{4}$)&\ds(1,4,5,-A$_{1}$)&\ds(1,44,13,A$_{18}$)&(2,19,12,A$_{20}$)&(2,67,20,A$_{1}$)\\[1pt]
\cline{3-3}\cline{5-6}
(0,25,10,A$_{4}$)&(0,61,16,A$_{2}$)&\ds(1,6,5,A$_{3}$)&(1,44,17,A$_{1}$)&\ds(2,20,10,-A$_{8}$)&\ds(2,71,18,A$_{21}$)\\[1pt]
\cline{4-6}
(0,25,12,A$_{6}$)&(0,61,16,-A$_{2}$)&(1,6,7,-A$_{1}$)&\ds(1,50,11,A$_{0}$)&\ds(2,23,8,A$_{0}$)&\ds(2,75,16,A$_{5}$)\\[1pt]
\cline{2-2}\cline{6-6}
(0,25,12,-A$_{6}$)&\ds(0,64,16,A$_{3}$)&(1,6,7,A$_{\infty}$)&(1,50,13,A$_{4}$)&(2,23,10,A$_{4}$)&\ds(2,79,18,A$_{21}$)\\[1pt]
\cline{3-3}\cline{6-6}
(0,25,12,A$_{22}$)&(0,64,18,A$_{22}$)&\ds(1,8,5,A$_{0}$)&(1,50,17,A$_{9}$)&(2,23,10,A$_{21}$)&\ds(2,80,18,A$_{3}$)\\[1pt]  
(0,25,12,A$_{\infty}$)&(0,64,18,A$_{\infty}$)&(1,8,7,A$_{4}$)&(1,50,19,A$_{1}$)&(2,23,14,A$_{20}$)&(2,80,20,A$_{\infty}$)\\[1pt]
\cline{2-2}\cline{4-6}
(0,25,14,A$_{1}$)&\ds(0,65,16,A$_{2}$)&(1,8,9,-A$_{1}$)&(1,56,15,A$_{3}$)&\ds(2,24,10,A$_{3}$)&\ds(2,83,14,A$_{0}$)\\[1pt]
\cline{3-3}
(0,25,14,-A$_{1}$)&(0,65,20,A$_{2}$)&\ds(1,10,7,A$_{9}$)&(1,56,15,A$_{7}$)&(2,24,12,A$_{\infty}$)&(2,83,16,A$_{4}$)\\[1pt]
\cline{1-1}\cline{5-5}
(0,29,16,A$_{1}$)&(0,65,20,-A$_{2}$)&(1,10,11,-A$_{1}$)&(1,56,17,A$_{\infty}$)&\ds(2,27,10,A$_{1}$)&(2,83,18,A$_{7}$)\\[1pt]
\cline{2-4}
\multicolumn{6}{|l|}~\\[1pt]
\hline
\end{tabular}
\\[1pt]
\caption{quadruples $(a,b,n,$matrix$)$ with $n\leq 20$}
\\
\begin{tabular}{|l|l|l|l|l|l|}
\hline
\multicolumn{6}{|l|}~\\[1pt]
\cline{2-2}\cline{3-3}\cline{6-6}
(2,83,20,A$_{21}$)
&\ds(3,58,15,A$_{6}$)
&\ds(4,60,16,A$_{3}$)
&(5,60,15,A$_{4}$)&(6,72,20,A$_{\infty}$)
&(7,104,17,A$_{0}$)\\[1pt]
\cline{1-2}\cline{4-5}
\ds(2,87,18,A$_{5}$)&\ds(3,64,15,A$_{5}$)
&(4,60,18,A$_{\infty}$)
&\ds(5,62,15,A$_{1}$)
&\ds(6,79,18,A$_{6}$)&(7,104,19,A$_{4}$)\\[1pt]
\cline{1-1}\cline{3-4}\cline{6-6}
\ds(2,95,20,A$_{21}$)&(3,64,17,A$_{1}$)
&\ds(4,61,20,A$_{20}$)&\ds(5,66,15,A$_{5}$)
&(6,79,20,A$_{20}$)&\ds(7,138,19,A$_{0}$)\\[1pt]
\cline{1-3}\cline{5-6}
(2,99,20,A$_{3}$)&\ds(3,68,13,A$_{0}$)
&(4,65,16,A$_{1}$)&(5,66,17,A$_{3}$)
&\ds(6,87,18,A$_{1}$)&\ds(8,20,12,A$_{3}$)\\[1pt]
\cline{5-5}
(2,99,20,A$_{7}$)&(3,68,15,A$_{4}$)
&(4,65,16,A$_{5}$)&(5,66,19,A$_{20}$)
&\ds(6,91,18,A$_{6}$)&(8,20,14,A$_{\infty}$)\\[1pt]
\cline{2-3}\cline{6-6}
(2,99,20,A$_{21}$)&\ds(3,70,17,A$_{3}$)
&\ds(4,69,16,A$_{6}$)&(5,66,19,A$_{\infty}$)
&(6,91,20,A$_{3}$)&\ds(8,21,12,A$_{5}$)\\[1pt]
\cline{1-1}\cline{3-6}
\ds(2,108,18,A$_{5}$)&(3,70,19,A$_{\infty}$)
&\ds(4,77,14,A$_{0}$)&\ds(5,68,17,A$_{6}$)
&\ds(6,95,16,A$_{0}$)&\ds(8,24,12,A$_{5}$)\\[1pt]
\cline{1-2}\cline{4-4}\cline{6-6}
\ds(2,111,16,A$_{0}$)&\ds(3,74,19,A$_{1}$)
&(4,77,16,A$_{4}$)&\ds(5,76,17,A$_{1}$)
&(6,95,18,A$_{4}$)&\ds(8,25,12,A$_{20}$)\\[1pt]
\cline{2-2}\cline{4-6}
(2,111,18,A$_{4}$)&\ds(3,76,17,A$_{5}$)
&(4,77,18,A$_{1}$)&\ds(5,80,17,A$_{6}$)
&\ds(6,103,20,A$_{1}$)&\ds(8,29,12,A$_{0}$)\\[1pt]
\cline{1-2}\cline{4-5}
\ds(2,119,20,-A$_{5}$)&\ds(3,88,19,A$_{3}$)
&(4,77,18,A$_{3}$)&\ds(5,84,19,A$_{3}$)
&\ds(6,127,18,A$_{0}$)&(8,29,14,A$_{4}$)\\[1pt]
\cline{1-2}\cline{4-4}\cline{6-6}
\ds(2,143,18,A$_{0}$)&\ds(3,94,15,A$_{0}$)
&(4,77,20,A$_{\infty}$)&\ds(5,86,15,A$_{0}$)
&(6,127,20,A$_{4}$)&\ds(8,33,14,A$_{3}$)\\[1pt]
\cline{3-3}\cline{5-5}
(2,143,20,A$_{4}$)&(3,94,17,A$_{4}$)
&\ds(4,80,16,A$_{5}$)&(5,86,17,A$_{4}$)
&\ds(6,163,20,A$_{0}$)&(8,33,16,A$_{\infty}$)\\[1pt]
\cline{1-1}\cline{3-6}
\ds(2,179,20,A$_{0}$)&(3,94,17,A$_{5}$)
&\ds(4,89,20,A$_{1}$)&\ds(5,90,19,A$_{1}$)
&\ds(7,18,11,A$_{3}$)&\ds(8,41,14,A$_{20}$)\\[1pt]
\cline{1-4}\cline{6-6}
(3,10,7,A$_{3}$)&\ds(3,124,17,A$_{0}$)
&\ds(4,96,20,A$_{3}$)&\ds(5,116,17,A$_{0}$)
&(7,18,13,A$_{\infty}$)&\ds(8,48,16,A$_{3}$)\\[1pt]
\cline{3-3}\cline{5-5}
(3,10,7,A$_{20}$)&(3,124,19,A$_{4}$)
&\ds(4,105,16,A$_{0}$)&(5,116,19,A$_{4}$)
&\ds(7,20,11,A$_{5}$)&(8,48,18,A$_{\infty}$)\\[1pt]
\cline{2-2}\cline{4-6}
(3,10,9,A$_{\infty}$)&\ds(3,158,19,A$_{0}$)
&(4,105,18,A$_{4}$)&\ds(5,150,19,A$_{0}$)
&\ds(7,22,11,A$_{20}$)&\ds(8,49,14,A$_{1}$)\\[1pt]
\cline{1-6}
\ds(3,14,7,A$_{0}$)&\ds(4,12,8,A$_{3}$)
&\ds(4,137,18,A$_{0}$)&\ds(6,16,10,A$_{3}$)
&\ds(7,26,11,A$_{0}$)&\ds(8,53,14,A$_{0}$)\\[1pt]
(3,14,9,A$_{4}$)&(4,12,10,A$_{\infty}$)&(4,137,20,A$_{4}$)&(6,16,12,A$_{\infty}$)&(7,26,13,A$_{4}$)&(8,53,16,A$_{4}$)\\[1pt]
\cline{1-6}
\ds(3,16,9,A$_{20}$)&\ds(4,13,8,A$_{20}$)
&\ds(4,173,20,A$_{0}$)&\ds(6,23,10,A$_{0}$)
&\ds(7,30,13,A$_{3}$)&\ds(8,56,16,-A$_{8}$)\\[1pt]
\cline{1-3}\cline{6-6}
\ds(3,18,9,A$_{3}$)&\ds(4,17,8,A$_{0}$)
&\ds(5,14,9,A$_{3}$)&(6,23,12,A$_{4}$)
&(7,30,15,A$_{\infty}$)&\ds(8,57,16,A$_{20}$)\\[1pt]
\cline{4-6}
(3,18,11,A$_{\infty}$)&(4,17,10,A$_{4}$)
&(5,14,11,A$_{\infty}$)&\ds(6,27,12,A$_{3}$)
&\ds(7,32,13,A$_{5}$)&\ds(8,65,18,A$_{3}$)\\[1pt]
\cline{1-3}\cline{5-5}
\ds(3,22,11,A$_{20}$)&(4,21,10,A$_{3}$)
&\ds(5,16,9,A$_{20}$)&(6,27,14,A$_{\infty}$)
&\ds(7,36,13,A$_{20}$)&(8,65,20,A$_{\infty}$)\\[1pt]
\cline{1-1}\cline{3-6}
\ds(3,24,9,A$_{1}$)&(4,21,10,A$_{20}$)
&\ds(5,20,9,A$_{0}$)&(6,31,12,A$_{5}$)
&\ds(7,38,13,A$_{5}$)&\ds(8,69,16,A$_{1}$)\\[1pt]
\cline{1-1}\cline{5-6}
\ds(3,26,11,-A$_{8}$)&(4,21,12,A$_{\infty}$)
&(5,20,11,A$_{4}$)&(6,31,12,A$_{20}$)
&\ds(7,44,13,A$_{1}$)&\ds(8,73,18,A$_{20}$)\\[1pt]
\cline{1-4}\cline{6-6}
\ds(3,28,9,A$_{0}$)&\ds(4,29,10,A$_{1}$)
&\ds(5,24,11,A$_{3}$)&\ds(6,39,12,A$_{1}$)
&(7,44,15,A$_{3}$)&\ds(8,81,16,A$_{0}$)\\[1pt]
\cline{4-4}
(3,28,11,A$_{3}$)&(4,29,12,A$_{20}$)
&(5,24,13,A$_{\infty}$)&\ds(6,40,14,A$_{3}$)
&(7,44,17,A$_{\infty}$)&(8,81,18,A$_{4}$)\\[1pt]
\cline{2-3}\cline{5-6}
(3,28,11,A$_{4}$)&(4,32,12,A$_{3}$)
&\ds(5,26,11,A$_{20}$)&(6,40,16,A$_{\infty}$)
&\ds(7,48,13,A$_{0}$)&\ds(8,84,20,A$_{3}$)\\[1pt]
\cline{3-4}\cline{6-6}
(3,28,13,A$_{20}$)&(4,32,12,-A$_{8}$)
&\ds(5,34,11,A$_{1}$)&\ds(6,43,12,A$_{0}$)
&(7,48,15,A$_{4}$)&\ds(8,89,18,A$_{1}$)\\[1pt]
\cline{3-3}\cline{5-5}
(3,28,13,A$_{\infty}$)&(4,32,14,A$_{\infty}$)&(5,36,13,A$_{3}$)&(6,43,14,A$_{4}$)&(7,50,15,-A$_{8}$)&(8,89,20,A$_{20}$)\\[1pt]
\cline{1-2}\cline{6-6}
\ds(3,34,11,A$_{1}$)&\ds(4,33,10,A$_{0}$)
&(5,36,13,A$_{20}$)&(6,43,14,A$_{5}$)
&(7,50,15,A$_{20}$)&\ds(8,101,20,A$_{6}$)\\[1pt]
\cline{5-6}
(3,34,15,A$_{20}$)&(4,33,12,A$_{4}$)
&(5,36,15,A$_{\infty}$)&(6,43,14,A$_{20}$)
&\ds(7,60,17,A$_{3}$)&\ds(8,109,20,A$_{1}$)\\[1pt]
\cline{1-4}\cline{6-6}
\ds(3,40,13,A$_{3}$)&\ds(4,37,14,A$_{20}$)
&\ds(5,38,11,A$_{0}$)&\ds(6,44,14,-A$_{8}$)
&(7,60,19,A$_{\infty}$)&\ds(8,113,18,A$_{0}$)\\[1pt]
\cline{2-2}\cline{4-5}
(3,40,15,A$_{\infty}$)&\ds(4,41,12,A$_{1}$)
&(5,38,13,A$_{4}$)&\ds(6,52,14,A$_{5}$)
&\ds(7,62,15,A$_{1}$)&(8,113,20,A$_{4}$)\\[1pt]
\cline{2-2}\cline{4-5}
(3,40,17,A$_{20}$)&\ds(4,45,14,A$_{3}$)
&(5,38,13,-A$_{8}$)&\ds(6,55,14,A$_{1}$)
&\ds(7,64,17,A$_{20}$)&(8,113,20,A$_{6}$)\\[1pt]
\cline{1-1}\cline{3-3}\cline{5-6}
\ds(3,44,13,A$_{1}$)&(4,45,16,A$_{20}$)
&\ds(5,42,13,A$_{5}$)&(6,55,16,A$_{3}$)
&\ds(7,74,15,A$_{1}$)&\ds(8,149,20,A$_{0}$)\\[1pt]
\cline{1-1}\cline{3-3}\cline{6-6}
\ds(3,46,11,A$_{0}$)&(4,45,16,A$_{\infty}$)
&\ds(5,46,15,A$_{20}$)&(6,55,16,A$_{20}$)
&(7,74,17,A$_{4}$)&\ds(9,22,13,A$_{3}$)\\[1pt]
\cline{2-3}\cline{5-5}
(3,46,13,A$_{4}$)&\ds(4,53,12,A$_{0}$)
&\ds(5,48,13,A$_{1}$)&(6,55,18,A$_{\infty}$)
&(7,78,19,A$_{3}$)&(9,22,15,A$_{\infty}$)\\[1pt]
\cline{3-4}\cline{6-6}
(3,46,15,A$_{6}$)&(4,53,14,A$_{1}$)
&\ds(5,50,15,A$_{3}$)&\ds(6,67,14,A$_{0}$)
&(7,78,19,A$_{20}$)&\ds(9,28,13,A$_{20}$)\\[1pt]
\cline{5-6}
(3,46,19,A$_{20}$)&(4,53,14,A$_{4}$)
&(5,50,17,A$_{\infty}$)&(6,67,16,A$_{4}$)
&\ds(7,80,17,A$_{1}$)&\ds(9,32,13,A$_{0}$)\\[1pt]
\cline{1-1}\cline{3-3}\cline{5-5}
(3,54,15,A$_{1}$)&(4,53,14,A$_{5}$)
&\ds(5,54,15,A$_{5}$)&(6,67,18,A$_{20}$)
&\ds(7,90,19,A$_{6}$)&(9,32,15,A$_{4}$)\\[1pt]
\cline{3-6}
(3,54,15,A$_{3}$)&(4,53,18,A$_{20}$)
&\ds(5,56,17,A$_{20}$)&\ds(6,71,16,A$_{1}$)
&\ds(7,98,19,A$_{1}$)&\ds(9,36,15,A$_{3}$)\\[1pt]
\cline{2-5}
(3,54,17,A$_{\infty}$)&\ds(4,57,16,A$_{6}$)
&\ds(5,60,13,A$_{0}$)&\ds(6,72,18,A$_{3}$)
&\ds(7,102,19,A$_{6}$)&(9,36,17,A$_{\infty}$)\\[1pt]
\cline{1-1}\cline{2-2}\cline{5-5}\cline{6-6}
 \multicolumn{6}{|l|}~\\[1pt]
\hline
\end{tabular}
\\[1pt]
\caption{quadruples $(a,b,n,$matrix$)$ with $n\leq 20$} 
\\[1pt]
\begin{tabular}{|l|l|l|l|l|l|}
\hline
\multicolumn{6}{|l|}~\\[1pt]
\cline{1-1}\cline{3-5}
\ds(9,46,15,A$_{20}$)& (10,24,16,A$_{\infty}$)
&\ds(10,91,20,A$_{20}$)&\ds(11,74,19,-A$_{8}$)
&\ds(12,85,20,A$_{20}$)&(14,47,20,A$_{4}$) \\[1pt]
\cline{1-6}
\ds(9,52,17,A$_{3}$)&\ds(10,31,14,A$_{20}$)
&\ds(10,95,18,A$_{0}$)&\ds(11,78,19,A$_{20}$)
&\ds(12,97,20,A$_{1}$)&\ds(14,51,20,A$_{3}$) \\[1pt]
\cline{2-2}\cline{4-6}          
(9,52,19,A$_{\infty}$)&\ds(10,35,14,A$_{0}$)
&\ds(10,95,20,A$_{4}$)&\ds(11,90,19,A$_{1}$)&  \ds(12,109,20,A$_{0}$)&\ds(14,71,20,A$_{20}$)\\[1pt]
\cline{1-1}\cline{3-6}
\ds(9,54,15,A$_{1}$)&(10,35,16,A$_{4}$)
&\ds(10,107,20,A$_{1}$)&\ds(11,102,19,A$_{0}$)
&\ds(13,30,17,A$_{3}$)&\ds(14,79,20,A$_{1}$)\\[1pt]
\cline{1-4}\cline{6-6}
\ds(9,58,15,A$_{0}$)&\ds(10,39,16,A$_{3}$)
&\ds(10,131,20,A$_{0}$)&\ds(12,28,16,A$_{3}$)
&(13,30,19,A$_{\infty}$)&\ds(14,83,20,A$_{0}$) \\[1pt]
\cline{3-3}\cline{5-6}
(9,58,17,A$_{4}$)&(10,39,18,A$_{\infty}$)
&\ds(11,26,15,A$_{3}$)&(12,28,18,A$_{\infty}$)
&\ds(13,40,17,A$_{20}$)&\ds(15,34,19,A$_{3}$) \\[1pt]
\cline{1-2}\cline{4-6}          
\ds(9,62,17,-A$_{8}$)&\ds(10,51,16,A$_{20}$)
&(11,26,17,A$_{\infty}$)&\ds(12,37,16,A$_{20}$)
&\ds(13,44,17,A$_{0}$)&\ds(15,46,19,A$_{20}$) \\[1pt]
\cline{1-4}\cline{6-6}          
\ds(9,64,17,A$_{20}$)&\ds(10,56,18,A$_{3}$)
&\ds(11,34,15,A$_{20}$)&\ds(12,41,16,A$_{0}$)
&(13,44,19,A$_{4}$)&\ds(15,50,19,A$_{0}$) \\[1pt]
\cline{1-1}\cline{3-3}\cline{5-6}          
\ds(9,70,19,A$_{3}$)&(10,56,20,A$_{\infty}$)
&\ds(11,38,15,A$_{0}$)&\ds(12,41,18,A$_{4}$)
&\ds(13,48,19,A$_{3}$)&\ds(16,36,20,A$_{3}$) \\[1pt]  
\cline{1-2}\cline{4-6}                   
\ds(9,76,17,A$_{1}$)&\ds(10,59,16,A$_{1}$)
&(11,38,17,A$_{4}$)&\ds(12,45,18,A$_{3}$)
&\ds(13,66,19,A$_{20}$)&\ds(16,49,20,A$_{20}$) \\[1pt] 
\cline{1-3}\cline{5-6}                   
\ds(9,82,19,A$_{20}$)&\ds(10,63,16,A$_{0}$)
&\ds(11,42,17,A$_{3}$)&(12,45,20,A$_{\infty}$)
&\ds(13,74,19,A$_{1}$)&\ds(16,53,20,A$_{0}$) \\[1pt]           
\cline{1-1}\cline{4-6}
\ds(9,88,17,A$_{0}$)&(10,63,18,A$_{4}$)
&(11,42,19,A$_{\infty}$)&\ds(12,61,18,A$_{20}$)
&\ds(13,78,19,A$_{0}$)& \\[1pt]
\cline{2-5}            
(9,88,19,A$_{4}$)&\ds(10,68,18,-A$_{8}$)
&\ds(11,56,17,A$_{20}$)&\ds(12,64,20,A$_{3}$)
&\ds(14,32,18,A$_{3}$)& \\[1pt]  
\cline{1-4}           
\ds(9,98,19,A$_{1}$)&\ds(10,71,18,A$_{20}$)
&\ds(11,60,19,A$_{3}$)&\ds(12,69,18,A$_{1}$)
&(14,32,20,A$_{\infty}$)& \\[1pt]
\cline{1-5}            
\ds(9,122,19,A$_{0}$)&\ds(10,75,20,A$_{3}$)
&\ds(11,64,17,A$_{1}$)&\ds(12,73,18,A$_{0}$)
&\ds(14,43,18,A$_{20}$)& \\[1pt]
\cline{1-5}                 
\ds(10,24,14,A$_{3}$)&\ds(10,83,18,A$_{1}$)
&\ds(11,68,17,A$_{0}$)&\ds(12,80,20,-A$_{8}$)
&\ds(14,47,18,A$_{0}$)&  \\[1pt]                
\cline{2-4}
\multicolumn{6}{|l|}~\\[1pt] 
\hline
\end{tabular}
\\[1pt]
\caption{quadruples $(a,b,n,$matrix$)$ with $n\leq 20$}\label{list}
\end{longtable}
\end{center}

\section{Spectral symmetry}\label{symmetry}

If a signed graph $G^\sigma$ in $\cal G'$ 
has a characteristic triple $(a,b,n)$ with $a=0$, then its spectrum is symmetric, which means that it is invariant under multiplication by $-1$.
In this case the negative $G^{-\sigma}$ is cospectral with $G^\sigma$.
When $G^{-\sigma}$ is switching isomorphic to $G^\sigma$,
we say that the signed graphs are {\em sign-symmetric}. 
The following result follows by straightforward verification.

\begin{prp}
A signed graph $G^\sigma\in\cal G'$ with no isolated edges is sign-symmetric 
if and only if its spectrum is symmetric and $G^{\sigma}$ nor its negative $G^{-\sigma}$ is
switching isomorphic to one of the following cases:
$A_1$,
$A_2$ unless $m=\ell$,
$A_5$,
$A_6$, 
$A_8$, 
$A_{13}$,
$A_{18}$,
$A_{19}$.
\end{prp}

Next we determine which signed graphs with symmetric spectrum are DSS.
It is clear that signed graphs with symmetric spectrum which are not sign symmetric, cannot be DSS.

\begin{thm}
Suppose $G^\sigma$ is a signed graph in $\cal G$ with a symmetric spectrum and no isolated edges.
Then $G^\sigma$ is DSS if and only if $G^{\sigma}$ or its nagative $G^{-\sigma}$ is switching isomorphic to one of the following:
$A_0(m,m)$,
$A_2(m,m)$ unless $m^2$ is a triangular number,
$A_3(m,m)$ unless $m=8$, $m=9$, or $(m+1)^2$ is the sum of two consequtive squares,
$A_{11}(3,4)$,
$A_{11}(4,3)$,
$A_{12}(4,4,4,4)$,
$A_{12}(6,3,3,6)$,
$A_{12}(6,6,3,3)$,
$A_{15}(4,4,2,2)$,
$A_{17}$.
\end{thm}

\begin{proof}
If $G^\sigma$ has matrix $A_0(m,m)$, $A_4(m,m)$, $A_{22}(m)$, or $A_\infty(m,m)$, and if $n\leq 20$ the result follows from Proposition~\ref{cosp} and Theorem~\ref{20}.
Moreover, $A_1(2,\ell)$ and $A_2(m,\ell)$ with $m\neq\ell$ are not sign symmetric, and therefore not DSS.
Thus we only need to consider $A_2(m,m)$, and $A_3(m,m)$.

First we look at $A_2(m,m)$ with characteristic triple
$(0,4m^2+1,4m)=(0,b,2\sqrt{b-1})$.
To find the signed graphs cospectral with $A_2(m,m)$
we need to find the characteristic triples $(0,b,n)$, where $b-1$ is an even square and $n\leq 2\sqrt{b-1}$.
By inspection we only find $A_{0}(k,k)$ when $2k^2-2k+1=4m^2+1$, that is $m^2= k(k-1)/2$.
In all remaining cases $A_2(m,m)$ is DSS.

Next we consider $A_3(m,m)$ with characteristic  triple
$(0,(m+1)^2,2m+2)=(0,b,2\sqrt{b})$.
Now we need to find the characteristic triples $(0,b,n)$, where $b$ is a square and $n\leq 2\sqrt{b}$.
We find $A_{12}(4,4,4,4)$, $A_{12}(6,6,3,3)$, which correspond to $m=8$ and $m=9$ respectively, and 
$A_0(k,k)$ when $2k^2-2k+1=(m+1)^2$, that is $(m+1)^2=k^2+(k-1)^2$.
In all remaining cases $A_3(m,m)$ is DSS.
\end{proof}

Matrix $A_{22}(m)$ represents the complete bipartite graph $K_{m,m}$ with the edges of a perfect matching removed. 
It is also the bipartite double of the complete graph $K_m$ denoted by $K_m\times K_2$.
As an unsigned graph $K_m\times K_2$ is determined by its spectrum,
but when considered as a signed graph this is far from true.
Indeed, for every $m\geq 3$ $A_{22}(m)$ is cospectral and not switching isomorphic with
$A_\infty(m,m)$, but there is more:

\begin{thm}
A signed graph $G^\sigma$ is cospectral with $K_m\times K_2$ if and only if 
$G^\sigma$ is switching isomorphic with one of the following signed graphs:
$A_\infty(m,m)$, 
$A_3(m-2,m-2)+K_2$, 
$\pm A_2(k,\ell)+(m-k-\ell)K_2$ when $(m-1)^2=4k\ell+1$,
$A_0(k,k)+(m-k)K_2$ when $(m-1)^2=k^2+(k-1)^2$,
$A_4(k,k)+(m-k-2)K_2$ when $(m-1)^2=k^2+(k+1)^2$, 
$\pm A_6(4,3)$ when $m=6$,
$A_{10}(3,4)$ when $m=7$,
$A_{10}(4,3)$ when $m=7$, 
$A_{12}(4,4,4,4)+2K_2$ when $m=10$, 
$A_{12}(6,6,3,3)+2K_2$ when $m=11$.
\end{thm}
\begin{proof}
The characteristic triple of $A_{22}(m)$ equals 
$(0,(m-1)^2,2m)=(0,b,2+2\sqrt{b})$.
Thus we need triples 
$(0,b,n)$ where $b$ is a square and $n\leq 2+2\sqrt{b}$.
By straightforward verification we obtain the described cases.
\end{proof}

\section{The friendship graph}\label{friendship}

The friendship graph $F_\ell$ is an unsigned graph that consists of 
$\ell$ edge disjoint triangles that meet in one vertex.
The friendship theorem states that $F_\ell$ is the unique graph with the property that every pair of distinct vertices has a unique common neighbor
(for this result neighbors are called friends).
It has been proved in~\cite{CHVW} that for $\ell\neq 16$ there is no unsigned graph cospectral with $F_\ell$. 
However, in this section we will see many signed graphs
which are cospectral, but not switching isomorphic with
$F_\ell$.

As a signed graph $F_\ell\in{\cal G}$ (when $\ell\geq 2$). 
The adjacency matrix is $A_{20}(1,\ell)$,
which is switching isomorphic with $A_1(1,\ell)$.
The chararteristic triple is equal to $(1,2\ell,2\ell+1)$.
 
\begin{thm}
There exists a signed graph cospectral 
but not switching isomorphic with the 
Friendship graph $F_\ell$ if and only if
$\ell\in\{12,18,30,39,54,60,75\}$,
$\ell\equiv 1\!\!\mod 3$, $\ell\equiv 1\!\!\mod 4$, 
$\ell$ is a square,
or a triangular number.
\end{thm}

\begin{proof}
Suppose $G^\sigma$ is a signed graph of order $n$ with no isolated edges, such that $G^\sigma+\alpha K_2$ is cospectral with $F_\ell$ for some $\alpha\geq 0$. 
Then $G^\sigma$ has characteristic triple
$(1,2\ell,n)$ with $n\leq 2\ell+1$.
From Table~\ref{classification} we obtain that
$G^\sigma$ must have one of the following 
characteristic triples:
\\[2pt]
\begin{tabular}{llll}
$(1,2(m-1)^2,2m-1)$ & $(1,6m+2,2m+3)$    & 
$(1,m(m+1),2m+1)$   & $(1,2(m+1)^2,2m+3)$\\ 
$(1,98,19)$         & $(1,120,21)$       & 
$(1,86,17)$         & $(1,108,19)$       \\
$(1,122,19)$        & $(1,150,21)$       &  
$(1,24,13)$         & $(1,36,13)$        \\ 
$(1,18m+2,2m+9)$    & $(1,14,9)$         & 
$(1,8m+2,2m+5)$     & $(1,92,17)$        \\ 
$(1,60,15)$         & $(1,62,15)$        & 
$(1,78,17)$         & $(1,44,13)$        \\ 
$(1,38,13)$         & $(1,32,13)$        &
$(1,(m-1)(m-2),2m-1)$. &                  \\
\end{tabular}
\\
Such a characteristic triple can 
be written as $(1,2\ell,n)$ for some integer $\ell$ and $n\leq 2\ell+1$, if and only if $\ell$ is as described above. 
\end{proof}

We remark that this result also appeared in \cite{WHL}. 
Unfortunately there the matrix $A_1$
with $m=3$ was overlooked, 
and hence the condition $\ell\equiv 1\!\!\mod 3$ is missing.

In none of the cases of the above proof,
the order $n$ is equal to $2\ell+1$. 
Thus we have:

\begin{cor}
Every signed graph cospectral
with $F_\ell$ is switching isomorphic to $F_\ell$, or disconnected.
\end{cor}

In other words, within the class of connected signed graphs every friendship graph is DSS.
\\

\noindent
{\bf\large Acknowledgement}
Part of the research was carried out when the second author visited The Netherlands on a grant from the mathematics cluster Diamant of the Dutch Research Council.


\begin{thebibliography}{9}

\bibitem{BCKW}
F. Belardo, S.M. Cioab\u{a}, J.H. Koolen and J. Wang,
Open problems in the spectral theory of signed graphs,
The Art of Discrete and Applied Mathematics {\bf 1} (2018), \#P2.10.

\bibitem{BH}
A.E. Brouwer and W.H. Haemers,
Spectra of Graphs,
Springer, 2012.

\bibitem{CHVW}
S.M. Cioab\u{a}, W.H. Haemers, J.R. Vermette and W. Wong,
The graphs with all but two eigenvalues equal to $\pm 1$,
J. Algebr. Comb. {\bf 41} (2015) 887-897.


\bibitem{GHMP}
E. Ghorbani, W.H. Haemers, H.R. Maimani and L. Parsaei Majd,
On sign-symmetric signed graphs,
Ars Mathematica Contemporanea {\bf 19} (2020), 83--93;
also: arXiv:2003.09981.

\bibitem{GS}
G.R.W. Greaves and Z. Stani\'c,
Signed (0,2)-graphs with few eigenvalues and a symmetric spectrum,
J. Comb. Designs {\bf 30} (2022), 332--353; 
also: arXiv:2107.11556.

\bibitem{HT}
W.H. Haemers and H. Topcu,
On signed graphs with at most two eigenvalues unequal to~$\pm 1$,
Linear Algebra and its Applications {\bf 670} (2023), 68-77.
Also: arXiv:2109.02522.

\bibitem{HT'}
W.H. Haemers and H. Topcu,
The signed graphs with two eigenvalues unequal to~$\pm 1$,
Applied Mathematics and Computation {\bf 463} (2024) 128348.
Also: arXiv:2301.01623.

\bibitem{WHL}
D. Wang, Y. Hou and D. Li,
Signed graphs with all but two eigenvalues equal to 
$\pm 1$, 
Graphs and Combinatorics {\bf 39} (2023).

\end{thebibliography}
\end{document}